\documentclass[12pt,reqno]{amsart}
\usepackage{amsfonts,amstext}
\usepackage{tikz}
\usepackage{amsmath,amstext,amssymb,amscd}
\usepackage{multicol}
\usepackage{calligra}
\usepackage{amsfonts}
\usepackage{amsthm}
\usepackage{amsmath}
\usepackage{amscd}
\usepackage[latin2]{inputenc}
\usepackage{t1enc}
\usepackage[mathscr]{eucal}
\usepackage{indentfirst}
\usepackage{graphicx}
\usepackage{graphics}
\usepackage{pict2e}
\usepackage{epic}
\numberwithin{equation}{section}
\usepackage[margin=2.9cm]{geometry}
\usepackage{epstopdf} 

\newtheorem{theorem}{Theorem}
\newtheorem{lemma}[theorem]{Lemma}

\newtheorem{example}{Example}

\DeclareMathAlphabet{\mathcalligra}{T1}{calligra}{m}{n}

\title [Ideal Containments under flat extensions]
{Ideal containments under flat extensions}
\author{Solomon Akesseh}

\begin{document}

\begin{abstract}
Let $\varphi : S =  k[y_0,..., y_n] \to R =  k[y_0,...,y_n]$ be given by $y_i \to f_i$ where $f_0,...,f_n$ is an $R$-regular sequence of homogeneous elements of the same degree. A recent paper shows for ideals, $I_\Delta \subseteq S$, of matroids, $\Delta$,  that $I_\Delta^{(m)} \subseteq I^r$ if and only if $\varphi_*(I_\Delta)^{(m)} \subseteq \varphi_*(I_\Delta)^r$ where $\varphi_*(I_\Delta)$ is the ideal generated in $R$ by $\varphi(I_\Delta)$. We prove this result for saturated homogeneous ideals $I$ of configurations of points in $\mathbb{P}^n$ and use it to obtain many new counterexamples to $I^{(rn - n + 1)} \subseteq I^r$ from previously known counterexamples.  
\end{abstract}
\maketitle

\section{Introduction}

\noindent Let $R$ be a commutative Noetherian domain. Let $I$ be an ideal in $R$. We define the $m$th symbolic power of $I$ to be the ideal $$ I^{(m)} = R \cap \bigcap_{P \in Ass_R(I)} I^mR_P \subseteq R_{(0)}.$$ In this note we shall be interested in symbolic powers of homogeneous ideals of $0$-dimensional subschemes in $\mathbb{P}^n$. In the case that the subscheme is reduced, the definition of the symbolic power takes a rather simple form by a theorem of Zariski and Nagata \cite{Eisenbud} and does not require passing to the localizations at various associated primes. Let $I \subseteq k[\mathbb{P}^n]$ be a homogeneous ideal of reduced points, $p_1,...,p_l$, in $\mathbb{P}^n$ with $k$ a field of any characteristic. Then $I = I(p_1) \cap \cdot \cdot \cdot \cap I(p_l)$ where $I(p_i) \subseteq k[\mathbb{P}^n]$ is the ideal generated by all forms vanishing at $p_i$, and the $m$th symbolic power of $I$ is simply $I^{(m)} = I(p_1)^m \cap \cdot \cdot \cdot \cap I(p_l)^m$. 
\\
\\
In \cite{ELS}, Ein, Lazarsfeld and Smith proved that if $I \subseteq k[\mathbb{P}^n]$ is the radical ideal of a $0$-dimensional subscheme of $\mathbb{P}^n$, where $k$ is an algebraically closed field of characteristic $0$, then $I^{(mr)} \subseteq (I^{(r +  1 - n)})^m$ for all $m \in \mathbb{N}$ and $r \geq n$. Letting $r = n$, we get that $I^{(mn)} \subseteq I^m$ for all $m \in \mathbb{N}$. Hochster and Huneke in \cite{HoHu} extended this result to all ideals $I \subseteq k[\mathbb{P}^n]$ over any field $k$ of arbitrary characteristic. 
\\
\\
In \cite{BH1} Bocci and Harbourne introduced a quantity $\rho(I)$, called the resurgence, associated to a nontrivial homogeneous ideal I in $k[\mathbb{P}^n]$, defined to be $\sup \{s/t : I^{(s)} \not \subseteq I^t \}$. It is seen immediately that if $\rho(I)$ exists, then for $s > \rho(I)t$, $I^{(s)} \subseteq I^t$. The results of \cite{ELS, HoHu} guarantee that $\rho(I)$ exists since $I^{(mn)} \subseteq I^m$ implies that $\rho(I) \leq n$ for an ideal $I$ in $k[\mathbb{P}^n]$. For an ideal $I$ of points in $\mathbb{P}^2$, $I^{(mn)} \subseteq I^m$ gives $I^{(4)} \subseteq I^2$. According to \cite{BH1} Huneke asked if $I^{(3)} \subseteq I^2$ for a homogeneous ideal $I$ of points in $\mathbb{P}^2$. More generally Harbourne conjectured in \cite{Primer} that if $I \subseteq k[\mathbb{P}^n]$ is a homogeneous ideal, then $I^{(rn - (n-1))} \subseteq I^r$ for all $r$. This led to the conjectures by Harbourne and Huneke in \cite{HaHu} for ideals I of points that $I^{(mn - n +1)} \subseteq \mathfrak{m}^{(m-1)(n-1)}I^m$ and $I^{(mn)} \subseteq \mathfrak{m}^{m(n-1)}I^m$ for $m \in \mathbb{N}$. 
\\
\\
The second conjecture remains open. Cooper, Embree, Ha and Hoeful give a counterexample in \cite{CEHH} to the first for $n = 2 = m$ for a homogeneous ideal $I \subseteq k[\mathbb{P}^2]$. The ideal $I$ in this case is $I = (xy^2, yz^2, zx^2, xyz) = (x^2, y) \cap (y^2, z) \cap (z^2, x)$ whose zero locus in $\mathbb{P}^2$ is the $3$ coordinate vertices of $\mathbb{P}^2$, $[0:0:1]$, $[0:1:0]$ and $[1:0:0]$ together with $3$ infinitely near points, one at each of  the vertices, for a total of $6$ points. Clearly the monomial $x^2y^2z^2 \in  (x^2, y)^3 \cap (y^2, z)^3 \cap (z^2, x)^3$ so $x^2y^2z^2$ is in $I^{(3)}$. Note $xyz \in I$ so $x^2y^2z^2 \in I^2$, but $x^2y^2z^2 \notin \mathfrak{m}I^2$. 
\\
\\
Shortly thereafter a counterexample to the containment $I^{(3)} \not\subseteq I^2$ was given by Dumnicki, Szemberg and Tutaj-Gasinska in \cite{DST}. In this case $I$ is the ideal of the $12$ points dual to the $12$ lines of the Hesse configuration. The Hesse configuration consists of the $9$ flex points of a smooth cubic and the $12$ lines through pairs of flexes. Thus $I$ defines $12$ points lying on $9$ lines. Each of the lines goes through $4$ of the points, and each point has $3$ of the lines going through it. Specifically $I$ is the saturated radical homogeneous ideal $I = (x(y^3 - z^3), y(x^3 - z^3), z(x^3 - y^3)) \subset \mathbb{C}[\mathbb{P}^2]$. Its zero locus is the $3$ coordinate vertices of $\mathbb{P}^2$ together with the $9$ intersection points of any $2$ of the forms $x^3 - y^3$, $x^3 - z^3$ and $y^3 - z^3$. The form $F = (x^3 - y^3)(x^3 - z^3)(y^3 - z^3)$ defining the $9$ lines belongs to $I^{(3)}$  since for each point in the configuration, $3$ of the lines in the zero locus of $F$ pass through the point, but $F \notin I^2$ and hence $I^{(3)} \not \subseteq I^2$. (Of course this also means that $I^{(3)} \not \subseteq \mathfrak{m}I^2$.) More generally, $I = (x(y^n - z^n), y(x^n - z^n), z(x^n - y^n))$ defines a configuration of $n^2 + 3$ points called a Fermat configuration \cite{Artebani}. For $n \geq 3$, we again have $I^{(3)} \not \subseteq I^2$ \cite{Harbourne-Seceleanu15, Seceleanu} over any field of characteristic not $2$ or $3$ containing $n$ distinct $n$th roots of $1$. 
\\
\\
Subsequent counterexamples to $I^3 \subseteq I^2$  were given in \cite{BCH}, \cite{BN}, \cite{Harbourne-Seceleanu15}, \cite{Czaplinski} and  \cite{Seceleanu} including related counterexamples to $I^{(nr-n+1)} \subseteq I^r$ for ideals of points in $\mathbb{P}^n$ in positive characterisitc given in \cite{Harbourne-Seceleanu15}. All of the counterexamples to $I^3 \subseteq I^2$ are ideals of points where the points are singular points of multiplicity at least $3$ of a configuration of lines. By considering flat morphisms $\mathbb{P}^n \to \mathbb{P}^n$, we obtain many new counterexamples to $I^{(rn - n +1)} \subseteq I^r$, taking $I$ to be the ideal of the fibers over the points of previously known counterexamples. 
\\
\\
The idea for this comes from \cite{Geramita15}. Suppose $\Delta$ is a matroid on $[s] = \{1,...,s\}$ of dimension $s-1-c$ and and let $f_1,...,f_s \in R = k[y_0,...,y_n]$ be homogeneous polynomials that form an R-regular sequence, $n \geq c$. Suppose now that $\varphi : S = k[y_1,..., y_s] \to R$ is a $k$-algebra map defined by $y_i \to f_i$. Then \cite{Geramita15} shows that if $I_\Delta \subseteq S$ is the ideal of the matroid and $m$ and $r$ are positive integers, then $I_\Delta^{(m)} \subseteq I_\Delta^r$ if and only if $\varphi_*(I_\Delta)^{(m)} \subseteq \varphi_*(I_\Delta)^r$ where $\varphi_*(I_\Delta)$ denotes the ideal generated by $\varphi(I_\Delta)$ in $R$. Of course a natural question is whether $I^{(m)} \subseteq I^r$ if and only if $\varphi_*(I)^{(m)} \subseteq \varphi_*(I)^r$ for any saturated homogeneous ideal. The current note answers this question in the affirmative for ideals $I$ of points in $\mathbb{P}^n$, relying on the ideas in \cite{Geramita15}.
\\
\\
I would like to thank Brian Harbourne, my adviser, for suggesting the idea of this note and providing guidance throughout its writing. I would like to thank Tom Marley, Alexandra Seceleanu and Mark Walker for helpful conversations.

\section{Results}

\noindent Throughout this note, let $R = S = k[y_0,...,y_n]$ and let $\{f_0,...,f_n\} \subseteq R$ be an R-regular sequence of homogeneous elements of $R$ of the same degree. Let $\varphi : S \to R$ be the $k$-algebra map given by $y_i \mapsto f_i$. For an ideal $I \subseteq S$, let $\varphi_*(I) \subseteq R$ denote the ideal generated by $\varphi(I)$. 

\begin{lemma}
Let $\varphi : S \to R$ be as above. Then $R$ is a free graded $S$-module, hence $R$ is faithfully flat as an $S$-module. 
\end{lemma}

\begin{proof}
It suffices to show that $R$ is free over $S$ since free modules are faithfully flat modules. Note that $\varphi$ is injective since $\{f_0,...,f_n\}$ is a regular sequence. It follows that $S \cong k[f_0,...,f_n] \subseteq R$. So we identify $S$ with $k[f_0,...,f_n]$ and show that $R$ is free over $k[f_0,...,f_n]$. Since $\{f_0,...,f_n\}$ is a maximal homogeneous $R$-regular sequence, it is a homogeneous system of parameters (sop). The reason is that every regular sequence is part of an sop and because $R$ is Cohen-Macaulay (CM), every sop is a regular sequence ($\text{depth} R = \dim R$) and so if $\{f_0,...,f_n\}$ is a maximal regular sequence, then it is an sop. Since $R = k[\mathbb{P}^n]$ is a positively graded affine $k$-algebra, the fact that $\{f_0,...,f_n\}$ is a homogeneous sop is equivalent to $R$ being a finite $S$-module by \cite[Theorem 1.5.17]{Bruns-Herzog}. Since both $R$ and $S$ are CM, $\text{depth} R = \dim R = n+1 = \dim S =  \text{depth} S$. By the Auslander-Buchsbaum formula \cite[Exercise 19.8]{Eisenbud} \cite[Theorem 15.3]{Peeva}, $\text{pd}_S R + \text{depth } R = \text{depth } S$. It follows that $\text{pd}_S R = 0$. So looking at the minimal free resolution of $R$ as an $S$-module, we see that $R$ is a free $S$-module. Therefore $R$ is a faithfully flat $S$-module.
\end{proof}

\begin{lemma}
Let $I \subseteq S$ be a homogeneous saturated  ideal defining a $0$-dimensional subscheme of $\mathbb{P}^n$. Then $\varphi_*(I) \subseteq R$ also defines a $0$-dimensional subscheme of $\mathbb{P}^n$. 
\end{lemma}

\begin{proof}
We start by showing that $R/\varphi_*(I)$ has the same Krull dimension as $S/I$. By the graded Auslander-Buchsbaum formula, $\text{pd}_S(R/\varphi_*(I))$ $+$ $\text{depth}(R/\varphi_*(I))$ $=$ $\text{depth} (S)$ $=$ $\text{pd}_S(S/I)$ $+$ $\text{depth}(S/I)$. By 3.1 in \cite{Geramita15}, $S/I$ and $R/\varphi_*(I)$ have the same graded Betti numbers so $\text{pd}_S(S/I) = \text{pd}_S(R/\varphi_*(I))$. Therefore $\text{depth}(S/I) = \text{depth}(R/\varphi_*(I))$. By 3.1 in \cite{Geramita15} again, $S/I$ is Cohen-Macaulay (CM) if and only if $R/\varphi_*(I)$ is CM. Since $I$ defines an ideal of points and is saturated, we have that $S/I$ is CM. It follows that $R/\varphi_*(I)$ is CM. For CM modules, the depth is the dimension so that $\dim S/I = \dim R/\varphi_*(I)$. Now since $S/I$ and $R/\varphi_*(I)$ are both CM, $\text{Ass}(R/\varphi_*(I))$ and $\text{Ass}(S/I)$ are both unmixed with their elements having height $\text{ht}(\varphi_*(I))$ and $\text{ht}(I)$ respectively. But $\text{ht}(\varphi_*(I)) = \text{ht}(I)$ since $\dim S/I = \dim R/\varphi_*(I)$. It follows that the elements of $\text{Ass}(R/\varphi_*(I))$ are all ideals of points. It follows that $\varphi_*(I)$ defines a $0$-dimensional subscheme of $\mathbb{P}^n$. 
\end{proof}

\begin{lemma}
Let $I \subseteq S$ be a saturated homogeneous ideal such that the zero locus of $I$ in $\mathbb{P}^n$ is $0$-dimensional. Let $\varphi : S \to R$ be as above. Then $\varphi_*(I^{(m)}) = \varphi_*(I)^{(m)}$. 
\end{lemma}

\begin{proof}
By Lemma 2, $\varphi_*(I)$  is the defining ideal of a $0$-dimensional subscheme so that $(\varphi_*(I))^{(m)} = \text{Sat}((\varphi_*(I))^m)$ where  $\text{Sat}((\varphi_*(I))^m)$ denotes the saturation of the ideal $(\varphi_*(I))^m$. An ideal and its saturation have the same graded homogeneous components for high enough degree so that for $t \gg 0$, $((\varphi_*(I))^{(m)})_t  = ((\varphi_*(I))^m)_t$. 
\\
\\
Using again that the symbolic power of an ideal of a $0$-dimensional subscheme in $\mathbb{P}^n$ is the saturation of the ordinary power, $I^{(m)} = \text{Sat}(I^m)$, we have that $(I^{(m)})_t = (I^m)_t$ for $t \gg 0$. Therefore $(\varphi_*(I^{(m)}))_t = (I^{(m)} \otimes_S R)_t = (I^m \otimes_S R)_t = (\varphi_*(I^m))_t$ for $t \gg 0$. Since $\varphi$ is a ring map, $\varphi_*(I^m) = (\varphi_*(I))^m$. This gives that $(\varphi_*(I^{(m)}))_t =  ((\varphi_*(I))^m)_t$ for $t \gg 0$. 
\\
\\
The last two paragraphs imply that $((\varphi_*(I))^{(m)})_t = \varphi_*(I^{(m)})_t $ for $t \gg 0$. Recall that $(\varphi_*(I))^{(m)}$ is saturated since it is the saturation of $(\varphi_*(I))^{m}$ and $ \varphi_*(I^{(m)})$ is saturated by Lemma 3.1 in \cite{Geramita15}. Two saturated graded homogeneous ideals that agree in degree $t$ for $t \gg 0$, agree in all degrees. Hence $(\varphi_*(I))^{(m)} = \varphi_*(I^{(m)})$.
\end{proof}

\begin{theorem}
Let $I \subseteq S$ be a saturated homogeneous ideal such that $V(I) \subseteq \mathbb{P}^n$ is a $0$-dimensional subscheme. Let $\varphi : S \to R$ be given by $y_i \to f_i$, $0 \leq i \leq n$, where $\{f_0,..., f_n\}$ is an $R$-regular sequence of homogeneous elements of $R$ of the same degree. Let $\varphi_{*}(I)$ denote the ideal in $R$ generated by $\varphi(I)$. Then $I^{(m)} \subseteq I^r$ if and only if $(\varphi_{*}(I))^{(m)} \subseteq (\varphi_{*}(I))^r $. 
\end{theorem}

\begin{proof}

$(\implies)$ Suppose that $I^{(m)} \subseteq I^r$. Then $\varphi(I^{(m)}) \subseteq \varphi(I^r)$ and so $\varphi_{*}(I^{(m)}) \subseteq \varphi_{*}(I^r)$. Since $\varphi$ is a homomorphism, $\varphi(I^r) = (\varphi(I))^r$. Note that $\varphi(I^r)$ generates $\varphi_{*}(I^r)$ in $R$ and $ (\varphi(I))^r$ generates $ (\varphi_{*}(I))^r$ in $R$. It follows that $\varphi_{*}(I^r) =  (\varphi_{*}(I))^r$ since they have the same generating set. Now applying Lemma 3 we have that $(\varphi_*(I))^{(m)} = \varphi_*(I^{(m)}) \subseteq \varphi_*(I^r) = \varphi_*(I)^r$ concluding the forward direction.
\\
\\
$(\Longleftarrow)$ Suppose now that for some homogeneous ideals $I$ and $J$ of $S$, $I \not \subseteq J$ but $\varphi_*(I) \subseteq \varphi_*(J)$. Then there is a homogeneous element $f \in I \backslash J$ such that $\varphi(f) \in \varphi_{*}(J)$. We may assume with no loss in generality that $I = (f)$. We have the sequence $$0 \to I \cap J \to I \oplus J \to I + J \to 0$$ with the first map given by $g \mapsto (g, -g)$ and the second map given by $(h, r) \mapsto h+r$.  It is clear that the sequence is exact. Since $\varphi$ is faithfully flat, we get an exact sequence $$0 \to \varphi_{*}(I \cap J) \to \varphi_{*}(I) \oplus \varphi_{*}(J) \to \varphi_{*}(I + J) \to 0.$$ Since $\varphi_{*}(I) \subseteq \varphi_{*}(J)$, $\varphi_{*}(I+J) = \varphi_{*}(J)$. Then the map $\varphi_{*}(I) \oplus \varphi_{*}(J) \to \varphi_{*}(J)$ has kernel $\varphi_{*}(I)$. It follows that $\varphi_{*}(I \cap J) = \varphi_{*}(I)$. This is impossible since the generators of $\varphi_{*}(I \cap J)$ are the images of the generators of $I \cap J$ and thus have degree greater than degree $f$ and hence greater than degree of $\varphi(f)$ which generates $\varphi_*(I) = I \otimes_S R \neq 0$. 
\\
\\
So it is the case that $\varphi(f) \notin \varphi_{*}(J)$. Hence $\varphi_{*}(I) \not \subseteq \varphi_{*}(J)$. Therefore if $I^{(m)} \not \subseteq I^r$, then by Lemma 3, $(\varphi_{*}(I))^{(m)} = \varphi_*(I^{(m)}) \not \subseteq (\varphi_{*}(I))^r$. Hence $(\varphi_{*}(I))^{(m)} \subseteq (\varphi_{*}(I))^r$ if and only if $I^{(m)} \subseteq I^r$. 
\end{proof}

\section{Examples}

\noindent Using the above result, we obtain many new counterexamples to the containment $I^{(3)} \subseteq I^2$ of ideals in $k[\mathbb{P}^2]$ and more generally counterexamples to the containment \begin{equation}\label{containment} I^{(nr-n+1)} \subseteq I^r \tag{$\star$} \end{equation} in $\mathbb{P}^n$. In particular if $I \subseteq k[\mathbb{P}^n]$ gives a counterexample to \eqref{containment}, then $\varphi_*(I)$  is a counterexample for any choice of homogeneous regular sequence $\{f_0,...,f_n\}$ of elements of the same degree. We illustrate this below with a few examples.

\begin{example}\rm
\noindent In this example, we work over $\mathbb{C}$. In \cite{DST}, the Fermat configuration, for $n=3$, was considered and its ideal $I = (x(y^3-z^3), y(x^3-z^3), z(x^3-y^3)) \subseteq \mathbb{C}[x,y,z]$ was found to be a counterexample to the containment $I^{(3)} \subseteq I^2$. Recall the configuration consists of the $3$ coordinate vertices and the $9$ intersection points of $y^3 - z^3$ and $x^3 - z^3$. The ideal $I$ is radical and all of the points in the configuration are reduced points. Now let $\varphi : \mathbb{C}[\mathbb{P}^2] \to \mathbb{C}[\mathbb{P}^2]$ by $x \to f = x^2 + y^2$, $y \to g = y^2 + z^2$ and $z \to h = x^2 + z^2$. One easily checks that $\{x^2 + y^2, y^2 + z^2, x^2 + z^2\}$ is a $\mathbb{C}[\mathbb{P}^2]$ - regular sequence. Then $\varphi$ induces a map of schemes $\varphi^{\#} : \mathbb{P}^2 \to \mathbb{P}^2$ which is faithfully flat. Consider the scheme-theoretic fibers of $\varphi^{\#}$ over the Fermat configuration and call it the fibered Fermat configuration. Note that the fibered Fermat configuration is $0$-dimensional. Since $\varphi^{\#}$ has degree $4$, the fibers consist of $48$ points of $\mathbb{P}^2$ where we count with multiplicity. The fibered Fermat configuration gives rise to the radical ideal $\varphi_*(I) = (f(g^3 - h^3), g(f^3 - h^3), h(f^3 - g^3)) \subseteq \mathbb{C}[\mathbb{P}^2]$ and by analyzing the ideal we see that the configuration consists of $4$ multiplicity $1$ points over each of the $3$ coordinate vertices, given by $f = 0 = g$, $f = 0 = h$ and $g = 0 = h$. The remaining $36$ points, each of multiplicity 1, in the configuration are the zero locus of $f^3 - h^3$ and $f^3 - g^3$. Since $I^{(3)} \not \subseteq I^2$, we have by Theorem 3 that $\varphi_*(I)^{(3)} \not \subseteq \varphi_*(I)^2$. 
\end{example}

\begin{example}\rm
\noindent We give another example of a fibered Fermat configuration whose ideal also gives a counterexample to the containment $I^{(3)} \subseteq I^2$. The difference here is that $36$ of the points in the configuration have multiplicity $1$ while the remaining $3$ points each has multiplicity $4$. So there are still $48$ points counting with multiplicity. Let $\varphi : \mathbb{C}[\mathbb{P}^2] \to \mathbb{C}[\mathbb{P}^2]$ by $x \to f = x^2$, $y \to g = y^2$ and $z \to h = z^2$. This faithfully flat ring map induces a morphism of schemes $\varphi^{\#} : \mathbb{P}^2 \to \mathbb{P}^2$ that is also flat. The fibers of $\varphi^{\#}$ over the Fermat configuration gives the fibered Fermat configuration that consists of the $36$ points, each of multiplicity $1$, of intersection of the degree $6$ forms $f^3 - g^3$ and $g^3 - h^3$. The configuration has $3$ more points each of multiplicity $4$ over the $3$ coordinate points. They are the zero loci of $f = 0 = g$, $f = 0 = h$ and $g = 0 = h$. So the fibered Fermat configuration here has points that are not all reduced. By Theorem 3, its nonradical ideal $\varphi_*(I)$ is a counterexample to the containment $\varphi_*(I)^{(3)} \subseteq \varphi_*(I)^2$.
\end{example}

\begin{example} \rm
\noindent Similarly for the Fermat configurations considered in \cite{Harbourne-Seceleanu15} for $n \geq 3$, we can construct new configurations of points, that may or may not be reduced in $\mathbb{P}^2$, that are the fibers of a morphism of schemes $\varphi^{\#}: \mathbb{P}^2 \to \mathbb{P}^2$. The morphism $\varphi^{\#}$ is induced by the ring map $\varphi: \mathbb{C}[\mathbb{P}^2] \to \mathbb{C}[\mathbb{P}^2]$ given by $x \to f$, $y \to g$ and $z \to h$ where $\{f,g,h\}$ is a homogeneous $\mathbb{C}[\mathbb{P}^2]$-regular sequence of the same degree. The Fermat configuration gives rise to a radical ideal $I = (x(y^j - z^j), y(x^j - z^j), z(x^j - y^j)) \subseteq \mathbb{C}[\mathbb{P}^2]$, $j \geq 3$, and for a choice of $\{f,g,h\}$, the fibered Fermat configuration gives rise to an ideal $\varphi_*(I) = (f(g^j - h^j), g(f^j - h^j), h(f^j - g^j))$, $j \geq 3$, not necessarily radical, that is also a counterexample to $\varphi_*(I)^{(3)} \subseteq \varphi_*(I)^2$. Here the Fermat configuration consists of the reduced $j^2$ points of intersection of $y^j - z^j$ and $x^j - y^j$ together with the $3$ coordinate vertices for a total of $j^2 + 3$ points. If the degree of the homogeneous elements in $\{f,g,h\}$ is $d$, then the fibered configuration consists of the $d^2j^2$ points of intersection of $g^j - h^j$ and $f^j - h^j$ together with the $3d^2$ fiber points over the three coordinate vertices that are the solutions of the three equations $f = 0 = g$, $f = 0 = h$ and $g = 0 = h$, counted with multiplicity. Again the points in the fibered configuration may or may not be reduced. 
\end{example}

\begin{example} \rm
\noindent Now we consider an example given in \cite{BCH} that is inspired by the example of the Fermat configuration. Let $k = \mathbb{Z}/3\mathbb{Z}$ and let $K$ be an algebraically closed field containing $k$. Note that $\mathbb{P}^2_{K}$ has $13$ $k$-points and $13$ $k$-lines such that each line contains $4$ of the points and each point is incident to $4$ of the lines. The forms $xy(x^2 - y^2)$, $xz(x^2 - z^2)$ and $yz(y^2 - z^2)$ vanish at all $13$ points of $\mathbb{P}^2_k$ but the form $x(x^2 - y^2)(x^2 - z^2)$ does not vanish at the point $[1:0:0]$. One checks easily that the ideal $I = (xy(x^2 - y^2), xz(x^2 - z^2), yz(y^2-z^2), x(x^2 - y^2)(x^2 - z^2)) \subseteq k[\mathbb{P}^2_{K}]$ is radical and its zero locus is the $13$ $k$-points of $\mathbb{P}^2_{K}$. Then $F = x(x-z)(x+z)(x^2 - y^2)((x- z)^2 - y^2)((x+z)^2 - y^2)$ defines $9$ lines meeting at $12$ points with each point incident to $3$ of the lines. It is not hard to see that $F \in I^{(3)}$ but $F \notin I^2$. So the reduced configuration that comes from $\mathbb{P}^2_k$ with the point $[1:0:0]$ removed together with all its incident lines gives rise to an ideal that is a counterexample to the containment $I^{(3)} \subseteq I^2$. Let $\varphi : k[\mathbb{P}^2_{K}] \to k[\mathbb{P}^2_{K}]$ be the ring map $x \to f = x^2$, $y \to g = y^2$ and $z \to h = z^2$. Applying the degree $4$ morphism of schemes $\varphi^{\#} : \mathbb{P}^2_{K} \to \mathbb{P}^2_{K}$, induced by $\varphi$,  and taking its fibers over the $k$-points, we get a configuration of $48$ points. For each point in the original configuration, we get $4$ points in the fibered configuration. The points in this new configuration are not all reduced. For instance over the point $[0:0:1]$, the fiber of $\varphi^{\#}$ is a point of multiplicity $4$ in $\mathbb{P}^2_{K}$ given by the vanishing of $y^2$ and $x^2$. The ideal of the fibered configuration as schemes is the ideal $\varphi_*(I) = (fg(f^2 - g^2), fh(f^2 - h^2), gh(g^2 - h^2), f(f^2 - g^2)(f^2 - h^2))$. This ideal is not radical and since $\{f,g,h\} \subset \mathbb{P}^2_{K}$ is a regular sequence, we have by Theorem 3 that $\varphi_*(I)^{(3)} \not \subseteq \varphi_*(I)^2$.  
If instead we take $f = x^2 + y^2$, $g = y^2 + z^2$ and $h = x^2 + z^2$ in the above example, then the fibered configuration we obtain is a reduced configuration and the ideal $\varphi_*(I)$ is a radical ideal satisfying $\varphi_*(I)^{(3)} \not \subseteq \varphi_*(I)^2$. 
\\
\\
\noindent Variations of the above example are considered in $\mathbb{P}^n$ for various $n$ in \cite{Harbourne-Seceleanu15}, giving counterexamples for the more general conjecture $I^{(nr - n +1)} \subseteq I^r$. We can apply our result to these to obtain new counterexamples to the more general containment.  

\end{example}

\bibliographystyle{plain}
\bibliography{IdealContainments}

\begin{thebibliography}{10}

\bibitem{Artebani}
Michela Artebani and Igor Dolgachev.
\newblock The {H}esse pencil of plane cubic curves.
\newblock {\em Enseign. Math. (2)}, 55(3-4):235--273, 2009.

\bibitem{BN}
Thomas Bauer, Sandra Di~Rocco, Brian Harbourne, Jack Huizenga, Anders Lundman,
  Piotr Pokora, and Tomasz Szemberg.
\newblock Bounded negativity and arrangements of lines.
\newblock {\em Int. Math. Res. Not. IMRN}, (19):9456--9471, 2015.

\bibitem{Primer}
Thomas Bauer, Sandra Di~Rocco, Brian Harbourne, Micha\l Kapustka, Andreas
  Knutsen, Wioletta Syzdek, and Tomasz Szemberg.
\newblock A primer on {S}eshadri constants.
\newblock In {\em Interactions of classical and numerical algebraic geometry},
  volume 496 of {\em Contemp. Math.}, pages 33--70. Amer. Math. Soc.,
  Providence, RI, 2009.

\bibitem{BCH}
Cristiano Bocci, Susan~M. Cooper, and Brian Harbourne.
\newblock Containment results for ideals of various configurations of points in
  {$\bold{P}^N$}.
\newblock {\em J. Pure Appl. Algebra}, 218(1):65--75, 2014.

\bibitem{BH1}
Cristiano Bocci and Brian Harbourne.
\newblock Comparing powers and symbolic powers of ideals.
\newblock {\em J. Algebraic Geom.}, 19(3):399--417, 2010.

\bibitem{Bruns-Herzog}
Winfried Bruns and J\"urgen Herzog.
\newblock {\em Cohen-{M}acaulay rings}, volume~39 of {\em Cambridge Studies in
  Advanced Mathematics}.
\newblock Cambridge University Press, Cambridge, 1993.

\bibitem{CEHH}
Susan~M. Cooper, Robert J.~D. Embree, Huy~T\`ai H\`a, and Andrew~H. Hoefel.
\newblock Symbolic powers of monomial ideals.
\newblock {\em Proc. Edinb. Math. Soc. (2)}, 60(1):39--55, 2016.

\bibitem{Czaplinski}
Adam Czapli\'nski, Agata G\l~\'owka, Grzegorz Malara, Magdalena
  Lampa-Baczy\'nska, Patrycja \L~uszcz \'Swidecka, Piotr Pokora, and Justyna
  Szpond.
\newblock A counterexample to the containment {$I^{(3)}\subset I^2$} over the
  reals.
\newblock {\em Adv. Geom.}, 16(1):77--82, 2016.

\bibitem{DST}
Marcin Dumnicki, Tomasz Szemberg, and Halszka Tutaj-Gasi\'nska.
\newblock Counterexamples to the {$I^{(3)}\subset I^2$} containment.
\newblock {\em J. Algebra}, 393:24--29, 2013.

\bibitem{ELS}
Lawrence Ein, Robert Lazarsfeld, and Karen~E. Smith.
\newblock Uniform bounds and symbolic powers on smooth varieties.
\newblock {\em Invent. Math.}, 144(2):241--252, 2001.

\bibitem{Eisenbud}
David Eisenbud.
\newblock {\em Commutative algebra}, volume 150 of {\em Graduate Texts in
  Mathematics}.
\newblock Springer-Verlag, New York, 1995.
\newblock With a view toward algebraic geometry.

\bibitem{Geramita15}
A.~V. Geramita, B.~Harbourne, J.~Migliore, and U.~Nagel.
\newblock Matroid configurations and symbolic powers of their ideals, 2015.

\bibitem{HaHu}
Brian Harbourne and Craig Huneke.
\newblock Are symbolic powers highly evolved?
\newblock {\em J. Ramanujan Math. Soc.}, 28A:247--266, 2013.

\bibitem{Harbourne-Seceleanu15}
Brian Harbourne and Alexandra Seceleanu.
\newblock Containment counterexamples for ideals of various configurations of
  points in {$\bold{P}^N$}.
\newblock {\em J. Pure Appl. Algebra}, 219(4):1062--1072, 2015.

\bibitem{HoHu}
Melvin Hochster and Craig Huneke.
\newblock Comparison of symbolic and ordinary powers of ideals.
\newblock {\em Invent. Math.}, 147(2):349--369, 2002.

\bibitem{Peeva}
Irena Peeva.
\newblock {\em Graded syzygies}, volume~14 of {\em Algebra and Applications}.
\newblock Springer-Verlag London, Ltd., London, 2011.

\bibitem{Seceleanu}
Alexandra Seceleanu.
\newblock A homological criterion for the containment between symbolic and
  ordinary powers of some ideals of points in {$\Bbb{P}^2$}.
\newblock {\em J. Pure Appl. Algebra}, 219(11):4857--4871, 2015.

\end{thebibliography}

\end{document}